\documentclass{amsart}

\usepackage{mathrsfs}
\usepackage{amsthm}
\usepackage{amscd}
\usepackage{amssymb}
\usepackage{latexsym}
\usepackage{graphicx}
\usepackage{stmaryrd}
\usepackage{xypic}
\usepackage{enumitem}
\xyoption{curve}
\usepackage{nicefrac}
\usepackage{wasysym}
\usepackage{enumitem}
\usepackage{setspace}
\usepackage{ifthen}
\usepackage{verbatim}
\usepackage{pigpen}
\usepackage[margin=1.5in]{geometry}

\newenvironment{frcseries}{\fontfamily{frc}\selectfont}{}

\usepackage{ulem}

\newcommand{\I}{\mathbb{I}}

\newcommand{\N}{\mathbb{N}}

\newcommand{\R}{\mathbb{R}}
\newcommand{\real}{\mathbb{R}}

\newcommand{\Z}{\mathbb{Z}}

\newcommand{\ra}{\rightarrow}

\newcommand{\xra}{\xrightarrow}

\newcommand{\ira}{\hookrightarrow}
\newcommand{\colim}{\mathrm{colim}}

\newcommand{\id}{1}
\newcommand{\im}{{\mathrm{im}\;\!\!}}

\newcommand{\MANIFOLDS}{\mathscr{M}}

\newcommand{\pushoutcorner}{\ar@{}[dr]|{\text{\pigpenfont R}}}
\newcommand{\pullbackcorner}{\ar@{}[dr]|{\text{\pigpenfont J}}}

\newcommand{\half}{\nicefrac{1}{2}}

\newcommand{\ring}[1]{{\ifthenelse{1=#1}{R}{{\ifthenelse{2=#1}{S}{T}}}}}
\newcommand{\semiring}[1]{{\ifthenelse{0=#1}{\N}{{\ifthenelse{1=#1}{S}{T}}}}}
\newcommand{\semimodule}[1]{{\ifthenelse{1=#1}{M}{N}}}

\newcommand{\tensor}[1]{{\ifthenelse{0=#1}{\otimes_\N\;}{{\ifthenelse{1=#1}{\otimes_{\semiring{1}}\;}{\otimes_{\semiring{1}[\Sigma_p]}}}}}}

\newcommand{\cohomologysheaf}[1]{\mathcal{H}^{#1}}

\newcommand{\openstar}{{\sf{st}}}
\newcommand{\sing}{{\sf{sing}}}

\newcommand{\digraph}[1]{\ifthenelse{1=#1}{G}{\ifthenelse{2=#1}{H}{I}}}

\newcommand{\constantsheaf}[2]{{\sf{Const}}_{#1;#2}}  

\newcommand{\globalsections}[2]{\Gamma_{#1}#2}
\newcommand{\sheaf}[1]{\ifthenelse{1=#1}{\mathcal{F}}{\ifthenelse{2=#1}{\mathcal{G}}{\mathcal{H}}}}
\newcommand{\cosheaf}[1]{\ifthenelse{1=#1}{\mathcal{A}}{\ifthenelse{2=#1}{\mathcal{B}}{\mathcal{C}}}}
\newcommand{\homologysheaf}{\mathcal{H}}
\newcommand{\Grading}{\bullet} 

\newcommand{\setmin}{{\,\! - \!\,}}

\newcommand{\alexanderduality}{{\sf AD}}  
\newcommand{\closed}{{\sf{clo}}}
\newcommand{\compact}{{\sf{com}}}
\newcommand{\proper}{{\sf{prop}}}
\newcommand{\pathcomponents}{\pi_0}
\newcommand{\projection}{\tau}
\newcommand{\freepositivecone}[1]{\R^+[#1]}
\newcommand{\sections}{{\sf{sec}}} 
\newcommand{\homology}[2]{H^{#1}_{#2}}
\newcommand{\bmhomology}[1]{H_{#1}}

\newcommand{\positivebmhomology}[1]{\;\!\!^+\!H_{#1}}

\newcommand{\cohomology}[2]{H^{#2}_{#1}}

\newcommand{\positivecohomology}[1]{\;\!\!^+\!H^{#1}}

\newcommand{\ambientspace}{E}
\newcommand{\sensedspace}{C}
\newcommand{\paramspace}{P}
\newcommand{\subparamspace}{Q}
\newcommand{\fiberspace}{\R^{n}}
\newcommand{\manifold}{M}

\newcommand{\timedomain}{\R}

\newcommand{\CONES}{{\sf{Cones}}}  

\newcommand{\positivecohomologysheaf}[1]{\;\!\!^+\!\mathcal{H}^{#1}}
\newcommand{\positivehomologysheaf}[1]{\;\!\!^+\!\mathcal{H}_{#1}}

\newtheorem*{thm:positive.alexander.duality}{Theorem \ref{thm:positive.alexander.duality}}
\newtheorem*{cor:criterion}{Corollary \ref{cor:criterion}}
\newtheorem*{cor:local}{Corollary \ref{cor:local}}

\newtheorem{thm}{Theorem}[section]
\newtheorem{cor}[thm]{Corollary}
\newtheorem{lem}[thm]{Lemma}
\newtheorem{prop}[thm]{Proposition}

\newtheorem*{lem*}{Lemma}
\newtheorem*{prob*}{Problem}

\theoremstyle{definition}
\newtheorem{defn}[thm]{Definition}
\newtheorem{eg}[thm]{Example}

\title{Positive Alexander duality for pursuit and evasion}
\author{Robert Ghrist}
\address{Departments of Mathematics \& Electrical/Systems Engineering, Univ. Pennsylvania}
\email{ghrist@math.upenn.edu}
\subjclass{Primary:55N30, 55S40, 55U30  Secondary: 57R19, 90C48, 91A44}
\thanks{Authors supported by US DoD contracts FA9550-12-1-0416 and N00014-16-1-2010.}

\author{Sanjeevi Krishnan}
\address{Department of Mathematics, Ohio State University}
\email{sanjeevi@math.osu.edu}

\begin{document}
\begin{abstract}
Considered is a class of pursuit-evasion games, in which an evader tries to avoid detection. Such games can be formulated as the search for sections to the complement of a coverage region in a Euclidean space over a timeline. Prior results give homological criteria for evasion in the general case that are not necessary {\it and} sufficient. This paper provides a necessary and sufficient {\it positive cohomological} criterion for evasion in a general case. The principal tools are (1) a refinement of the \v{C}ech cohomology of a coverage region with a positive cone encoding spatial orientation, (2) a refinement of the Borel-Moore homology of the coverage gaps with a positive cone encoding time orientation, and (3) a positive variant of Alexander Duality.  Positive cohomology decomposes as the global sections of a sheaf of local positive cohomology over the time axis; we show how this decomposition makes positive cohomology computable as a linear program. 
\end{abstract}
\maketitle
\section{Introduction}
\label{sec:intro}
The motivation for this paper comes from a type of \textit{pursuit-evasion game}. In such games, two classes of agents, \textit{pursuers} and \textit{evaders} move in a fixed geometric domain over time. The goal of a pursuer is to capture an evader (by, e.g., physical proximity or line-of-sight). The goal of an evader is to move in such a manner so as to avoid capture by any pursuer. This paper solves a feasibility problem of whether an evader can win in a particular setting under certain constraints.

We specialize to the setting of pursuers-as-sensors, in which, at each time, a certain region of space is ``sensed'' and any evader in this region is considered captured. Evasion, the successful evasion of all pursuers by an evader, corresponds to gaps in the sensed region over time. We formalize this setting in the form of a \textit{coverage pair} $(\ambientspace,\sensedspace)$ over the time-axis $\timedomain$: one has, for each time $t$, the ambient space $\ambientspace_t\cong\real^n$ and a sensed or \textit{covered} region $\sensedspace_t\subset\ambientspace_t$. The ambient space and covered regions at each point in time are bundled into an ambient space-time $\ambientspace\cong\real^n\times\real$ and a covered subspace $\sensedspace\subset\ambientspace$.
The goal of the evader is to construct an \textit{evasion path}, a section to the restriction $\ambientspace\setmin\sensedspace\ra\timedomain$ of the projection map $\ambientspace\ra\timedomain$. 

This is not necessarily a topological problem. However, epistemic restrictions on $\sensedspace$ are natural in real-world settings. For example, a distributed sensor network may not directly perceive its geometric coverage region $\sensedspace$ but can often infer its (co)homology from local computations. The present paper focuses on the feasibility of evasion: 

\begin{prob*}
  Is there an ``evasion criterion'' for the existence of a section to the complement
  \begin{equation*}
    \ambientspace\setmin\sensedspace\ra\timedomain,
  \end{equation*}
  based on a coordinate-free description of $\sensedspace$?
\end{prob*}

Work of Adams and Carlsson gives a geometric criterion for the plane ($n=2$, see below); there is no other known sharp criterion. The problem is one of topological duality. 
While the (timewise) \textit{unstable} homotopy theory of $\ambientspace\setmin\sensedspace$ determines the existence of an evasion path, currently only the less informative (timewise) \textit{stable} homotopy theory of $\ambientspace\setmin\sensedspace$ admits a duality with the (timewise) stable homotopy theory of the input space $\sensedspace$.

\subsection{Background}
Different formulations of pursuit-evasion games suggest different methods of analysis. Combinatorial descriptions of the domain suggest methods from the theory of graphs \cite{ParsonsPursuit78} and cellular automata \cite{BCKHigh08,CGCyclic14}. Geometric descriptions of agent behavior suggest methods from differential equations and differential game theory \cite{IsaacsDifferential65}, computational geometry \cite{GLL+visibility99,SGallSolution01}, probability theory \cite{HKSMultiple99,VSK+Probabilistic02}, and Alexandrov geometry \cite{ABGCapture09,ABGTotal10}. Coordinate-free descriptions of pursuer behavior model the data available to ad-hoc networks of non-localized sensors, to which coordinate geometry is unavailable; such descriptions suggest methods from topology \cite{ACEvasion14,SGCoordinate06,SGCoverage07}.

The first topological criterion for evasion is based on the homologies of the fibers $\sensedspace_t$ relative their topological boundaries in $\ambientspace_t$ \cite{SGCoordinate06}.
Assume for convenience that coverage gaps are uniformly bounded in the sense that there is a sufficiently large, closed ball $B\subset\fiberspace$ such that $\fiberspace\setmin B\subset C_t$ at each time $t\in\timedomain$. The topological boundary $F$ of this ball $B$ in $\R^n$ encircles all potential losses of coverage like a {\it fence}; this fence naturally determines a homology class $[F]$ in $H_{n-1}\sensedspace$. The homological criterion says that if there is an evasion path, then $[F]\neq 0$. This criterion for evasion, while necessary, is not sufficient [Figure \ref{fig:conservative}].

\begin{figure}[h]
\begin{center}
  \includegraphics[width=5.5in]{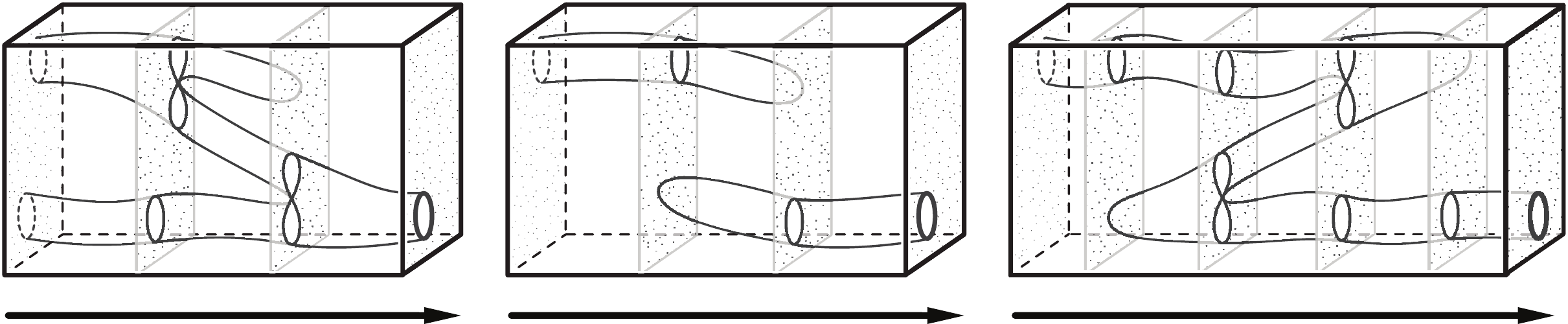} 
  \caption{Three coverage pairs over $\timedomain$, all planar examples ($n=2$). Evasion is possible in the first case only [left]. The simple homological criterion rules out evasion in the second [middle], but not third [right], case.}
  \label{fig:conservative}
\end{center}
\end{figure}

Another topological criterion for evasion follows from Alexander Duality (classical or parametrized \cite{Kalisnik-VerovsekAlexander13}), a duality between the cohomology of the fiberwise coverage regions $\sensedspace_t$ and the homology of the complementary coverage gaps $\ambientspace_t\setmin\sensedspace_t$ \cite{ACEvasion14}. One such criterion states that if there is an evasion path, then there exists a persistent barcode in the {\it zigzag persistent homology} \cite{CSZigzag10} of $\sensedspace$. This criterion for evasion, while necessary, is not sufficient [Figure \ref{fig:conservative}, right]. The main problem, as observed in \cite{ACEvasion14}, is that the fiberwise homotopy type of $\sensedspace$ fails to detect enough information about the embedding $\sensedspace\ira\ambientspace$, and hence $\ambientspace\setmin\sensedspace$, to deduce evasion: see Figures \ref{fig:works} and \ref{fig:noworks}.

A necessary \textit{and} sufficient refinement of the previous criterion exists for the special case where $n=2$ and $\sensedspace_t$ is a connected union of balls at all times $t$ that only changes topology at finitely many points $t$ in time \cite{ACEvasion14}. In this case, an algorithm iteratively constructs an evasion path at each time $t$ of topology change (via a Reeb graph) based on the affine and orientation structures $\sensedspace_t$ inherits from $\R^{n}$. An open question in that paper is whether there is a generalization that does not require information about affine structure.

\subsection{Contributions}
We give a necessary, sufficient, and tractable cohomological criterion for evasion that affirmatively answers this open question in a general case and arbitrary dimension. The (co)homology theories of parametrized spaces used in this paper, at least in the cases considered in the main results, are equivalent to real Borel-Moore homology having proper supports and real \v{C}ech cohomology.
A positive cone $\positivebmhomology{q}\paramspace$ inside the homology $\bmhomology{q}\paramspace$ of a space $\paramspace$ over $\R^q$, reminiscent and sometimes a special case of cones in the homology of a \textit{directed space} \cite{grandis2004inequilogical}, encodes some information about the parametrization. A positive cone $\positivecohomology{k}\paramspace$ inside the cohomology $\cohomology{}{k}\paramspace$ of a space $\paramspace$ over $\R^q$, whose fibers are endowed with the additional structure of pro-objects in a category of oriented $k$-manifolds, encodes some information about those orientations. The main result is the following positive variant of Alexander Duality.

\begin{thm:positive.alexander.duality}
  For coverage pair $(\ambientspace,\sensedspace)$ over $\R^q$ with $\dim E>2$ and $0\leqslant q\leqslant\dim\ambientspace-2$,
  \begin{equation*}
    \positivecohomology{\dim\ambientspace-q-1}\sensedspace\cong\positivebmhomology{q}(\ambientspace\setmin\sensedspace)
  \end{equation*}
\end{thm:positive.alexander.duality}

Along the way, we investigate properties of positive (co)homology.
We show that positive homology: (1) preserves certain limits [Lemma \ref{lem:positive-exactness}] and colimits [Lemma \ref{lem:proper.supports}] and (2) characterizes the existence of sections [Lemmas \ref{lem:sections.determine.positivity}, \ref{lem:positivity.determines.sections}] in degree $1$.
We define positive cohomology in stages, for: (1) oriented manifolds-with-compact-boundaries; then (2) pro-objects thereof; and finally (3) parametrized spaces whose fibers have the additional structure of pro-objects of oriented manifolds-with-compact-boundaries.
Consequently, we conclude a complete criterion for evasion by considering Positive Alexander Duality in homological degree $q=1$.

\begin{cor:criterion}
  For a coverage pair $(\ambientspace,\sensedspace)$ over $\R$ with $\dim E>2$,
  \begin{equation}
    \positivecohomology{\dim\ambientspace-2}\sensedspace\neq\varnothing
  \end{equation}
  if and only if the restriction $\ambientspace\setmin\sensedspace\ra\timedomain$ of the projection $\ambientspace\ra\timedomain$ admits a section.
\end{cor:criterion}

This complete criterion for evasion assumes a minimum of information about the coverage region.
Unlike \cite{ACEvasion14}, we neither assume that our coverage region has the homotopy type of a CW complex, nor restrict the ambient spatial dimension of the region to $\dim\ambientspace_t=\dim\ambientspace-1=2$, nor restrict the manner in which the topology of our coverage region changes in time. In addition, the criterion is only based on the \v{C}ech cohomology and orientation data of the coverage region.

Moreover, this evasion criterion is efficiently computable. 
The positive cohomology of a coverage region over $\R$ decomposes as the global sections of a sheaf of local positive cohomology cones.

\begin{cor:local}
  For coverage pair $(\ambientspace,\sensedspace)$ over $\timedomain$ and collection $\mathscr{O}$ of open subsets of $\timedomain$,
  \begin{equation*}
    \positivecohomology{\dim\ambientspace-2}\sensedspace=\lim_{U\in\mathscr{O}}\positivecohomology{\dim\ambientspace-2}\sensedspace_U.
  \end{equation*}
\end{cor:local}

Abstractly, the global sections of a cellular sheaf of finitely generated, positive cones over a finitely stratified space is computable as a linear programming problem. We can therefore demonstrate the tractability of the evasion criterion with a prototypical pair of examples in \S\ref{sec:examples}.

We use the language of sheaf theory throughout, as part of a broader goal in introducing sheaf-theoretic methods in applied and computational topology. 
Below we fix some conventions in notation and terminology regarding sheaves, spaces, and cones.


\subsubsection{Topology}
Let $X$ denote a space.
We write $\closed(X)$ and $\compact(X)$ for the respective sets of closed and compact subsets of $X$.
We write $\pathcomponents{X}$ for the set of path-components natural in $X$.
For $B\subset X$ and a collection $\Psi$ of subsets of $X$, we write $\Psi|B$ and $\Psi\cap B$ for
\begin{equation*}
  \Psi|B=\{A\in\Psi\;|\;A\subset B\},\quad \Psi\cap B=\{A\cap B\;|\;A\in\Psi\}
\end{equation*}

A \textit{manifold-with-compact-boundary} is a manifold-with-boundary whose boundary is compact.
Let $\MANIFOLDS_n$ denote the category of non-compact, connected, oriented manifolds-with-compact boundary and oriented embeddings $\psi:M\ra N$ between such oriented manifolds with $\psi(M)\subset N\setmin\partial N$.

A \textit{space over $X$} is a space $\paramspace$ equipped with a map $\projection_\paramspace\colon\paramspace\ra X$. 
A \textit{section} to a space $\paramspace$ over $X$ is a section to $\projection_\paramspace$. For in inclusion $\iota:U\ira V$ of spaces and a space $\paramspace$ over $V$, we write $\paramspace_U$ for the space over $U$ defined by the pullback square
\begin{equation*}
  \xymatrix{
      \paramspace_U\ar[r]\ar[d]_{\projection_{\paramspace_U}}\pullbackcorner & \paramspace\ar[d]^{\projection_\paramspace}
    \\
      U\ar[r]_\iota &
      V
  }
\end{equation*}
A \textit{pair of spaces over $X$} is a pair $(\paramspace,\subparamspace)$ of spaces over $X$ with $\subparamspace$ a subspace of $\paramspace$ and $\projection_\subparamspace$ the pullback of $\projection_\paramspace$ along inclusion $\subparamspace\ira\paramspace$.
A space $\paramspace$ over $X$ is \textit{proper} if $\projection_\paramspace$ is proper.
A \textit{cube complex over $\R^q$} is a locally finite union $\paramspace$ of isothetic (axis-aligned) compact hyperrectangles in $\R^{n+q}$ with $\projection_\paramspace$ projection onto the last $q$ coordinates.
For a space $\paramspace$ over $X$, let $\sections_\paramspace$ denote the sheaf on $X$ sending each open $U\subset X$ to the set of sections to $\paramspace_U$.

\subsubsection{Sheaves}
Fix a space $X$ and a collection $\Psi$ of closed subsets of $X$ closed under intersection.
Let $\constantsheaf{X}{V}$ denote the constant sheaf on $X$ taking values in an object $V$.
Let $\sheaf{1}_{(A)}$ denote the pullback of a sheaf $\sheaf{1}$ on $X$ to $A\subset X$.
Let $\globalsections{}\sheaf{1}$ denote the global sections of a sheaf $\sheaf{1}$.

We fix some constructions and notation for Abelian sheaf theory; refer to 
\cite{bredon2012sheaf} for details.
Consider a sheaf $\sheaf{1}$ of real vector spaces on $X$.
Let $\globalsections{\Psi}\sheaf{1}$ denote the global sections of $\sheaf{1}$ whose supports lie in the collection $\Psi$ of subsets of the base space.
For an injective resolution $\mathcal{I}_\Grading$ of $\sheaf{1}$, we write $\homology{\Psi}{\Grading}$ and $\cohomology{\Psi}{\Grading}$ for
\begin{align*}
    \homology{\Psi}{\Grading}\sheaf{1}&=\homology{}{\Grading}\globalsections{\Psi}(U\mapsto(\globalsections{\compact|U}{\mathcal{I}_\Grading})^*)\\
    \cohomology{\Psi}{\Grading}\sheaf{1}&=\cohomology{}{\Grading}\globalsections{\Psi}{\mathcal{I}_\Grading},
\end{align*}

For simplicity, we do not define relative sheaf (co)homology and instead restrict our attention to pairs $(X,A)$ on which the relative theory reduces to an absolute theory.
Fix closed $A\subset X$ and suppose $\Psi$ is a collection of paracompact subsets of $X$ which each admit a neighborhood in $X$ which is an element in $\Psi$.
There exists a long exact sequence
\begin{equation}
  \label{eqn:cohomological.les}
     \cdots\ra
    \cohomology{\Psi}{\Grading}\sheaf{1}
      \ra
    \cohomology{\Psi\cap A}{\Grading}\sheaf{1}_{(A)}
      \xra{\partial^\Grading}
    \cohomology{\Psi|X\setmin A}{\Grading+1}\sheaf{1}_{(X\setmin A)}=\cohomology{\Psi|X\setmin A}{\Grading+1}\sheaf{1}
      \ra\cdots,
\end{equation}
\cite[Proposition 12.3]{bredon2012sheaf}.
If additionally $X$ has finite cohomological dimension (in the sense that the compactly supported cohomology of open subsets of $X$ vanishes for large enough degrees \cite{bredon2012sheaf}), then there exists a long exact sequence
\begin{equation}
  \label{eqn:homological.les}
     \cdots\ra
    \homology{\Psi}{\Grading}\sheaf{1}
      \ra
    \homology{\Psi\cap(X\setmin A)}{\Grading}\sheaf{1}_{(X\setmin A)}=\homology{\Psi\cap(X\setmin A)}{\Grading}\sheaf{1}
      \xra{\partial_{\Grading-1}}
    \homology{\Psi|A}{\Grading-1}\sheaf{1}_{(A)}
      \ra\cdots
\end{equation}
\cite[(95),\S13.2]{bredon2012sheaf}.
In both exact sequences, the unlabelled arrows are evident inclusions and restrictions.

For $X$ an oriented $n$-manifold in $\Psi$, we write $\Delta$ for Poincar\'{e} Duality
\begin{equation*}
    \Delta:\cohomology{\Psi}{n-p}\sheaf{1}=\homology{\Psi}{p}\sheaf{1},\quad q=0,1,\ldots,n
\end{equation*}

\subsubsection{Cones}
In this paper, a \textit{cone} is a subset of a real vector space $V$ closed under vector addition and closed under scalar multiplication by $\R^+$, the positive real numbers. A cone $K$ in $V$ is \textit{positive} if $K\cap -K=\varnothing$. 
%
%
%
Let $\CONES$ denote the category whose objects are cones and whose morphisms are functions between cones that are restrictions and corestrictions of linear maps between vector spaces generated by the cones. 

\begin{eg}
  For real vector spaces $V$ and $W$,
  \begin{equation*}
      \CONES(V,W)
  \end{equation*}
  is the set of linear maps $V\ra W$ and hence can be regarded itself as a real vector space with vector addition and scalar multiplication defined point-wise.
\end{eg}

For cones $V^+,W^+$ in real vector spaces $V,W$ with $V$ generated by $V^+$, we regard
\begin{equation}
  \label{eqn:cone-hom}
  \CONES(V^+,W^+)\subset\CONES(V,W)
\end{equation}
as a cone inside the vector space $\CONES(V,W)$ under the injection $\CONES(V^+,W^+)\ra\CONES(V,W)$ sending a cone map to its unique linear extension.

\subsection{The Formal Problem}
The setting for pursuit-evasion in this paper is a \textit{coverage pair}.

\begin{defn}
  \label{defn:coverage.pair}
  \textit{A coverage pair over a space $T$} is a pair $(\ambientspace,\sensedspace)$ of spaces over $T$ such that:
  \begin{enumerate}
    \item $\ambientspace$ is a real vector bundle over $T$; and
    \item $\sensedspace$ is a closed and connected subspace of $\ambientspace$; and
    \item The closure of $\ambientspace\setmin\sensedspace$ in $\ambientspace$ is proper.
  \end{enumerate}
\end{defn}

We make some immediate remarks on the definition.
Firstly, we stress the connectedness implicit in the definition of a coverage pair: as noted in \cite{ACEvasion14}, this is a necessary condition for evasion criteria. 
Secondly, the ambient space $\ambientspace$ of a coverage pair $(\ambientspace,\sensedspace)$ over a contractible space $T$ is homeomorphic to $\R^n\times T$; for the case $T=\R^q$, we always take $\ambientspace$ to be an oriented manifold.
Thirdly, the motivating example of $T$ is $T=\timedomain$ --- the setting for {\em temporal} evasion problems. To the extent possible, we develop our machinery for more general base spaces with an eye towards future uses in evasion-type problems over directed or partially-directed parameter spaces or state spaces.

We wish to regard the fibers of $\sensedspace$ as oriented.
Hence we generalize orientations for general spaces as the structure of \textit{pro-objects} in a category of oriented manifolds.
We identify each fiber $\sensedspace_t$ of a coverage region $\sensedspace$ in a coverage pair $(\ambientspace,\sensedspace)$ over $T$ with the pro-object in $\MANIFOLDS_{\dim\ambientspace_t}$ represented by the inverse system of closed neighborhoods of $\sensedspace_t$ in $\ambientspace_t$ which are objects in $\MANIFOLDS_{\dim\ambientspace_t}$.

\section{Homology}\label{sec:homology}

In this section, let $\paramspace$ denote a space with finite cohomological dimension over another space such that each proper subspace of $\paramspace$ is paracompact and admits a proper neighborhood in $\paramspace$. 
A motivating example of $\paramspace$ is the complement $\ambientspace\setmin\sensedspace$ of the coverage region in the ambient space of a coverage pair $(\ambientspace,\sensedspace)$ over Euclidean space.
For each such $\paramspace$, let $\proper(\paramspace)$ denote the collection of proper subspaces of $\paramspace$ and define
\begin{equation*}
  \bmhomology{\Grading}\paramspace=\homology{\proper(\paramspace)}{\Grading}(\constantsheaf{\paramspace}{\R}).
\end{equation*}
(This variant of homology differ subtly from the usual singular and Borel-Moore theories.)

\begin{eg}
  For a cube complex $\paramspace$ over $\R$ and real numbers $x<y$,
  \begin{equation*}
    \bmhomology{\Grading}\paramspace_{(x,y)}=\homology{\sing}{\Grading}(\paramspace_{[x,y]},\paramspace{\{x,y\}}),
  \end{equation*}
  where $\homology{\sing}{\Grading}$ denotes real singular homology.
\end{eg}

The construction $\bmhomology{\Grading}$ is contravariant in maps of the form $\paramspace_U\ira\paramspace$, for $U$ an open subspace of the codomain of $\projection_\paramspace$, and covariant in maps of the form $\paramspace_A\ira\paramspace$, for $A$ a closed subspace of the codomain of $\projection_\paramspace$.
For $U$ an open subspace of the space to which $\paramspace$ maps, the long exact sequence (\ref{eqn:homological.les}) specializes to the following long exact sequence, in which the unlabelled arrows are induced from inclusions and restrictions.
\begin{equation*}
  \cdots\ra\bmhomology{k}\paramspace_U\xra{\partial_k}\bmhomology{k-1}(\paramspace\setmin \paramspace_U)\xra{\bmhomology{k-1}(\paramspace\setmin\paramspace_U\ira \paramspace)}\bmhomology{k-1}\paramspace\ra\cdots
\end{equation*}

\subsection{Positive homology definition}

\begin{figure}[h]
\begin{center}
  \includegraphics[width=3.75in]{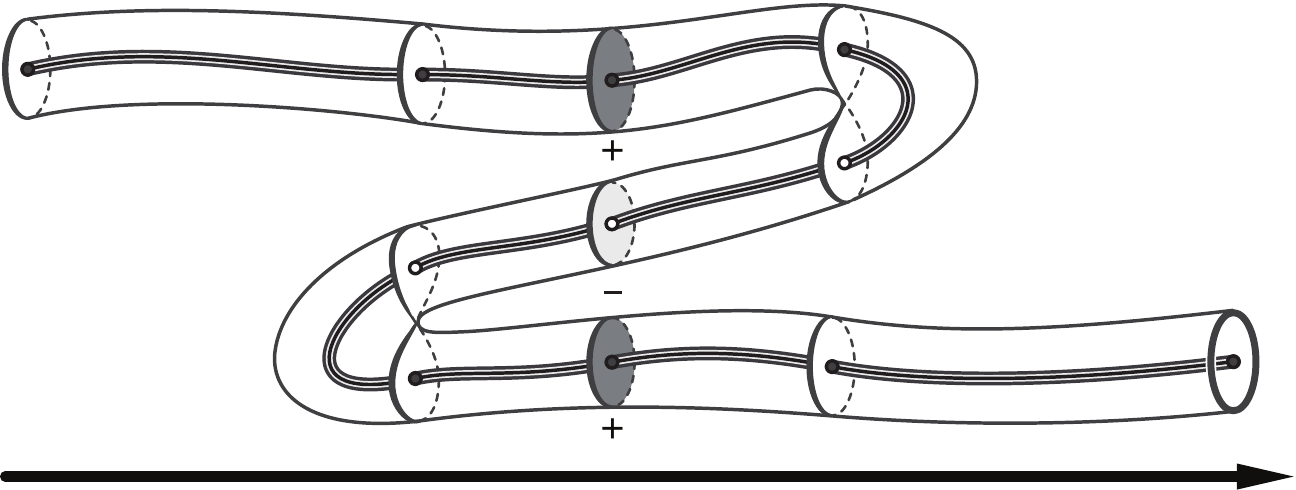} 
  \caption{{\bf Non-Positive Homology Class} The path at the core of the solid tube $\paramspace$ over $\timedomain$ represents an element in $\bmhomology{1}\paramspace$ that does not lie in the positive cone $\positivebmhomology{1}\paramspace$ because the path intersects a fiber negatively.}
  \label{fig:positive.bmhomology}
\end{center}
\end{figure}

Our goal is to define a positive homology that remembers directedness: see Figure \ref{fig:positive.bmhomology}.  To that end, we observe the following convoluted reinterpretation of ordinary $H_0$ of a space $X$ over a point:
\begin{equation*}
  \bmhomology{0}X=\colim_{\compact\ K\subset X}\CONES(\globalsections{}{(\constantsheaf{K}{\R})},\R) .
\end{equation*}

This points to a simple modification for definition positive homology in degree zero:

\begin{defn}
  \label{defn:positive-homology-zero}
  Define $\positivebmhomology{0}\paramspace$ to be the cone
  \begin{equation*}
    \positivebmhomology{0}\paramspace=\colim_{\compact\ K\subset\paramspace}\CONES(\globalsections{}{(\constantsheaf{K}{\R^+})},\R^+)
  \end{equation*}
  natural in spaces $\paramspace$ (over $\R^0$).
\end{defn}

Recall that we can regard $\positivebmhomology{0}\paramspace$ to be a cone inside $\bmhomology{0}\paramspace$ by (\ref{eqn:cone-hom}).

\begin{eg}
  For a CW complex $\paramspace$ over a point, $\positivebmhomology{0}\paramspace$ is the positive cone generated by 
  \begin{equation*}
    \pi_0\paramspace\subset\R[\pi_0\paramspace]=\bmhomology{0}\paramspace.
  \end{equation*}
  and hence remembers unstable homotopical information about $\paramspace$ inside the stable homotopical invariant $\bmhomology{0}\paramspace$.
\end{eg}

We may then define positive homology in arbitrary degree over more general base spaces (open subsets of a directed $\R^q$) inductively, using pullback squares.  
Given a space $\paramspace$ over an open $U\subset\R^{q}$, we write $\paramspace_t$ for the space over $\R^{q-1}$ such that $\projection_{\paramspace_t}$ is the composite
$$\paramspace_{\{t\}\times \R^q}\ra\{t\}\times\R^{q-1}\ra\R^{q-1}$$
of a restriction and corestriction of $\projection_\paramspace$ with projection $\{t\}\times\R^{q-1}\ra\R^{q-1}$ onto the last $q-1$ factors and $\paramspace_{(-\infty,t)}$ for $\paramspace_{U\cap((-\infty,t)\times\R^{q-1})}$.

\begin{defn}
  \label{defn:positive-homology}
  For positive integers $q$, define $\positivebmhomology{q}\paramspace$ by the pullback square
  \begin{equation}
    \label{eqn:positive-homology}
    \xymatrix@C=12pc{
        **[l]\positivebmhomology{q}\paramspace\ar[r]\ar[d]\pullbackcorner
      & **[r]\prod_{t}\positivebmhomology{q-1}\paramspace_{t}\ar[d]
      \\
        **[l]\bmhomology{q}\paramspace\ar[r]_-{\prod_{t}\bmhomology{q}\paramspace\xra{\bmhomology{q}(\paramspace_{(-\infty,t)}\subset\paramspace)}\bmhomology{q}\paramspace_{(-\infty,t)}\xra{\partial_q}\bmhomology{q-1}\paramspace_t}
      & **[r]\prod_{t}\bmhomology{q-1}\paramspace_{t},
    }
  \end{equation}
  where $t$ denotes an element in $\R$ such that $\{t\}\times\R^{q-1}\cap U\neq\varnothing$, natural in spaces $\paramspace$ over open $U\subset\R^q$, in $\CONES$.
\end{defn}

\begin{eg}
  We are particularly interested in $\positivebmhomology{1}$. In this case, (\ref{eqn:positive-homology})
  is easy to interpret for $\paramspace$ an open subspace of a vector bundle over $\R$: a positive homology class is represented by a locally finite singular $1$-cycle, whose $1$-chains are oriented in the sense that their composites with $\projection_\paramspace$ are non-decreasing maps $\I\ra\timedomain$.
  Figure \ref{fig:positivehomologyexamples} illustrates four examples of $\positivebmhomology{1}$ of a 1-d space over $\R$.  
\end{eg}

\begin{figure}[h]
\begin{center}
  \includegraphics[width=5.5in]{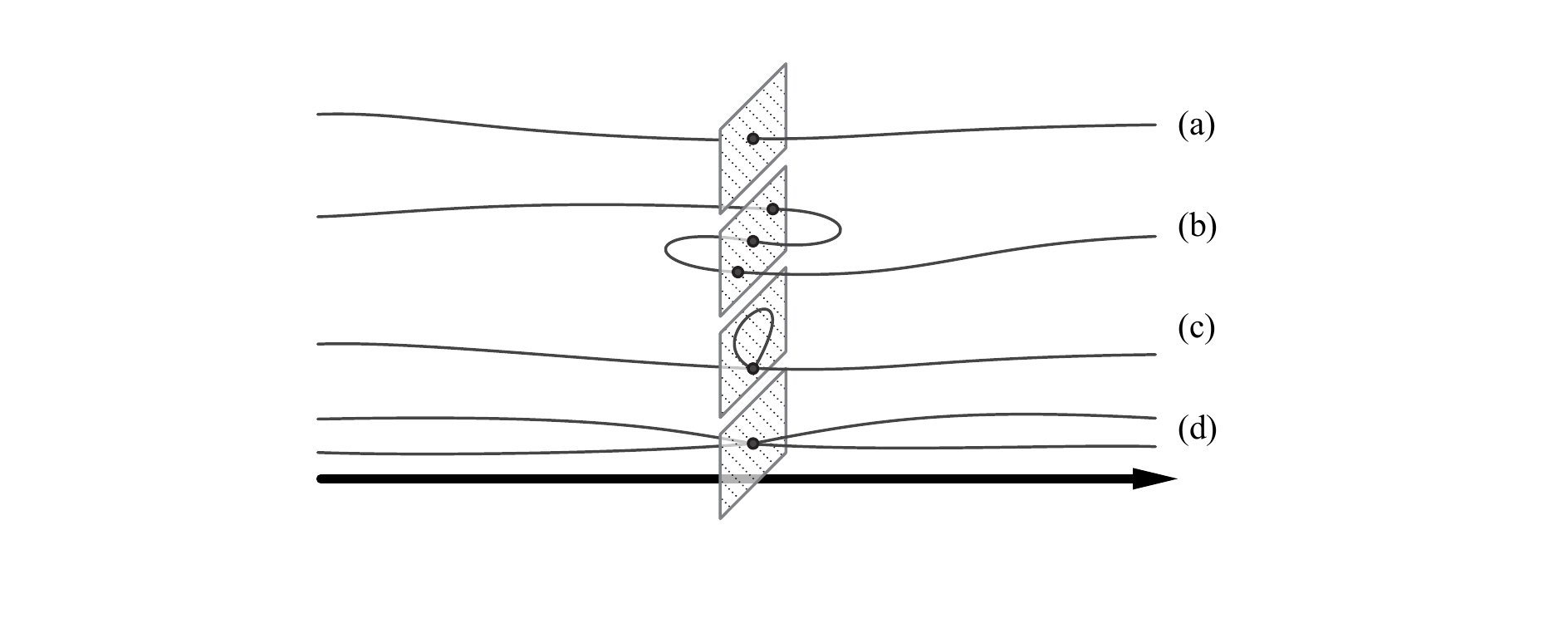} 
  \caption{{\bf Positive 1-homology cones:} Four simple examples of 1-d spaces over $\R$. (a) $\bmhomology{1}\cong\R$ and $\positivebmhomology{1}\cong\R^+$. (b) Here, also, $\bmhomology{1}\cong\R$, but, as there is a reversal in the path, $\positivebmhomology{1}=\emptyset$. (c) With the wedged circle, $\bmhomology{1}\cong\R\times\R$ and $\positivebmhomology{1}\cong\R^+\times\R$. (d) In this last example, there are four distinct sections over the base space and $\positivebmhomology{1}$ is a cone with four vectors as a basis. However, $\bmhomology{1}\cong\R^3$, since only three are linearly independent. Thus $\positivebmhomology{1}$ is a cone in $\R^3$ generated by four vectors.}
  \label{fig:positivehomologyexamples}
\end{center}
\end{figure}

\begin{eg}
  \label{eg:spheres}
  For the space $B$ over $\R^q$ such that $\projection_B=\id_{\R^q}$, $\positivebmhomology{q}B=\R^+$.
\end{eg}

A homology cone is always positive.
Hence we call $\positivebmhomology{\Grading}\paramspace$ the \textit{positive homology} of $\paramspace$.
Positive homology inherits from $\bmhomology{\Grading}$ a covariance in closed inclusions of spaces over $\timedomain$ and contravariance in inclusions of the form $\paramspace_U\ira\paramspace_V$ for spaces $\paramspace$ over $\timedomain$ and open $U\subset V\subset\timedomain$.

\subsection{Positive homology limits and colimits}

In the most relevant case of degree one homology of spaces over $\R$, $\positivebmhomology{1}$ defines a sheaf over $\timedomain$.

\begin{lem}
  \label{lem:positive-exactness}
  For a space $\paramspace$ over $\timedomain$ and collection $\mathscr{O}$ of open subsets of $\timedomain$,
  \begin{equation*}
    \positivebmhomology{1}\paramspace=\lim_{U\in\mathscr{O}}\positivebmhomology{1}\paramspace_U
  \end{equation*}
\end{lem}
\begin{proof}
  Let $\positivehomologysheaf{0},\homologysheaf_{0},\positivehomologysheaf{1},\homologysheaf_{1}$ be the presheaves on $\timedomain$ naturally defined by
  \begin{align*}
  \positivehomologysheaf{0}(U)&=\prod_{x\in U}\positivebmhomology{0}\paramspace_x,
  \;
  &\positivehomologysheaf{1}(U)&=\positivebmhomology{1}\paramspace_U, 
  \\
  \homologysheaf_0(U)&=\prod_{x\in U}\bmhomology{0}\paramspace_x,
  \;
  &\homologysheaf_1(U)&=\bmhomology{1}\paramspace_U.
  \end{align*}

  Let $\partial$ denote the sheaf map $\homologysheaf_{1}\ra\homologysheaf_{0}$ such that
  \begin{equation*}
    \partial_U=\prod_{x\in\timedomain}(\bmhomology{1}\paramspace_{U\cap(-\infty,x)}\xra{\partial_1}\bmhomology{0}\paramspace_x)\circ\bmhomology{1}(\paramspace_{U\cap(-\infty,x)}\subset\paramspace).
  \end{equation*}  
  
  There exists an isomorphism of presheaves
  \begin{equation}
    \label{eqn:positive.homology.sheaf}
    \homologysheaf_1\times_\partial\positivehomologysheaf{0}\cong\positivehomologysheaf{1}.
  \end{equation}

  The presheaves $\positivehomologysheaf{0}$ and $\homologysheaf_{0}$ are sheaves because products commute with limits.
  The presheaf $\homologysheaf_1$ is a sheaf by an application of the First Fundamental Theorem of Sheaves \cite[Theorem IV.2.1]{bredon2012sheaf}.
  Hence the left side of (\ref{eqn:positive.homology.sheaf}), and hence also the right side of (\ref{eqn:positive.homology.sheaf}), is a sheaf because section-wise limits of sheaves are sheaves.
\end{proof}

\begin{lem}
  \label{lem:proper.supports}
  Consider an open subspace $\paramspace$ of a vector bundle over $\timedomain$.
  Then
  \begin{equation}
    \label{eqn:proper.supports}
    \positivebmhomology{1}\paramspace=\colim_{K}\positivebmhomology{1}K,
  \end{equation}
  where the colimit is taken over all proper cube complexes $K\subset\paramspace$ over $\timedomain$.
\end{lem}
\begin{proof}
  We take $\paramspace$ to be an open subspace of $\R^{n+1}$, with $\projection_\paramspace$ projection onto the last coordinate, without loss of generality by the triviality of all bundles over $\R$. 
  We conflate $\bmhomology{\Grading}$ with real singular homology on locally contractible subspaces of $\paramspace$, and in particular write $\bmhomology{1}(\paramspace_x,K_x)$ to denote the relative real singular homology of the pair $(\paramspace_x,K_x)$ for $K$ a cube complex.

  Consider a proper cube complex $L\subset\paramspace$ over $\R$.
  It suffices to construct a proper cube complex $M\subset\paramspace$ over $\R$ containing $L$ such that for each $x\in\R$, the restriction of $(M_x\subset\paramspace_x)_*:\bmhomology{0}M_x\ra\bmhomology{0}\paramspace_x$ to the image of $(L_x\subset M_x)_*:\bmhomology{0}L_x\ra\bmhomology{0}M_x$ is injective.
  For then the top horizontal arrow in (\ref{eqn:positive-homology}) (for the case $U=\timedomain$) restricts and corestricts to a dotted arrow making the diagram
  \begin{equation*}
    \xymatrix{
        **[l]\positivebmhomology{1}\paramspace\ar@{.>}[r]\ar[d]\pullbackcorner
      & **[r]\colim_{\text{proper }K\subset\paramspace}\prod_{x\in \R}\positivebmhomology{0}K_x\ar[d]
      \\
        **[l]\colim_{\text{proper }K\subset\paramspace}\bmhomology{1}K
        \ar[r]
      & **[r]\colim_{\text{proper }K\subset\paramspace}\prod_{x\in U}\bmhomology{0}K_{x},
    }
  \end{equation*}
  whose vertical arrows are inclusions and whose bottom horizontal arrow is a corestriction of the bottom horizontal arrow in (\ref{eqn:positive-homology}) (for the case $U=\R$), commute and hence define a pullback square, because proper cube complexes inside $\paramspace$ are cofinal among all proper subspaces of $\paramspace$.
  The lemma would then follow because pullbacks commute with filtered colimits in $\CONES$.
  
  Consider $x\in\timedomain$. 
  The image of $\partial_1:\bmhomology{1}(\paramspace_x,L_x)\ra\bmhomology{0}L_x$ is finitely generated by $L_x$ compact and hence the image of relative $1$-homology classes supported on a compact set of the form $F_x\times\{x\}\subset\paramspace_x$, which we can take to 1) contain $L_x$ by $L_x$ compact and 2) be a finite union of isothetic, compact hyperrectangles because such hyperrectangles form a neighborhood basis of each point in $\paramspace_x$.
  There exist a compact neighborhood $I_x$ of $x$ in $\timedomain$ such that $F_x\times I_x\subset\paramspace$ by each point in $\paramspace$ admitting a neighborhood basis of compact hyperrectangles.
  
  Each closed cell $c\subset\timedomain$ admits a finite cover by sets of the form $I_x$ for $x\in c$ by $c$ compact.
  A union $M$ of sets of the form $I_x\times F_x$, over the sets $I_x$ in all such finite covers over all closed cells in $\timedomain$, defines a proper cube complex containing $L$.
  Inside
  \begin{equation}
    \label{eqn:ses}
    \xymatrix@C=5pc{
        \bmhomology{1}(\paramspace_x,L_x)
        \ar[r]^-{\partial_1}
      & \bmhomology{0}L_x
        \ar@{.>}[d]^-{\bmhomology{0}(L_x\subset M_x)}
        \ar[r]^-{(L_x\ira\paramspace_x)_*}
      & \bmhomology{0}\paramspace_x        
      \\
        \bmhomology{1}(M_x,L_x)
        \ar[r]_-{\bmhomology{0}(L_x\ira M_x)\circ\partial_1}
        \ar[u]^{\bmhomology{1}((\paramspace_x,L_x)\subset(\paramspace_x,M_x))}
      & \bmhomology{0}L_x\ar@{.>}[r]_-{\bmhomology{0}(M_x\subset\paramspace_x)}
      & \bmhomology{0}M_x
        \ar[u],
    }
  \end{equation}  
  the left vertical arrow is surjective by our choice of $M_x$ and hence the image of the top left horizontal arrow coincides with the image of the bottom left horizontal arrow, hence the kernel of the top right horizontal arrow coincides with the kernel of the bottom right horizontal arrow.
  Hence the restriction of the rightmost vertical arrow to the image of the bottom right arrow is injective.
\end{proof} 

\subsection{Detecting sections}

Sections to a space $\paramspace$ over $\R^q$ determine positive elements in $\bmhomology{q}\paramspace$.

\begin{lem}
  \label{lem:sections.determine.positivity}
  For a space $\paramspace$ over $\R^q$ admitting a section, $\positivebmhomology{q}\paramspace\neq\varnothing$.
\end{lem}
\begin{proof}
  A section to $\paramspace$ defines a closed inclusion $B\ra\paramspace$ of spaces over $R^q$, with $B$ the space over $\R^q$ defined by $\projection_B=\id_{\R^q}$, and hence induces a cone map $\positivebmhomology{q}B\ra\positivebmhomology{q}\paramspace$ with domain isomorphic to $\R^+$ and hence non-empty.
\end{proof}

A converse holds for homological degree $1$, under some point-set conditions.

\begin{lem}
  \label{lem:positivity.determines.sections}
  An open subspace $\paramspace$ of a vector bundle over $\timedomain$ admits a section if $\positivebmhomology{1}\paramspace\neq\varnothing$.
\end{lem}
\begin{proof}
  We take $\paramspace$ to be an open subspace of the bundle $\R^{n+1}$ over $\R$, whose bundle map is projection onto the last factor, without loss of generality by the triviality of vector bundles over $\timedomain$.  

  Suppose $\positivebmhomology{1}\paramspace\neq\varnothing$.
  There exists a proper cube complex $K\subset\paramspace$ such that $\positivebmhomology{0}K\neq\varnothing$ [Lemma \ref{lem:proper.supports}].
  We can endow $\timedomain$ with the structure of a CW complex such that $K_e$ is the product of a compact hyperrectangle with $e$, for each open edge $e\subset\timedomain$, by $K$ a cube complex.  
  Let $\sigma_t$ be the composite linear map
  \begin{equation*}
    \bmhomology{1}K\xra{\bmhomology{1}(K_{(-\infty,t)}\subset K)}\bmhomology{1}K_{(-\infty,t)}\xra{\partial_1}\bmhomology{0}K_t.
  \end{equation*}
 
  Commutative squares of the following form, where $c$ denotes a closed cell in $\timedomain$ having a point $t_c$ in the associated open cell, define the dotted arrows inside of them because the right vertical arrows are isomorphisms by our choice of CW structre on $\timedomain$.
  \begin{equation*}
    \xymatrix{
        **[l]\positivebmhomology{0}K_{t_c}\ar@{=}[r]
      & **[r]\freepositivecone{\pi_0K_{t_c}}\\
        **[l]\positivebmhomology{1}K\ar[u]^{\sigma_{t_c}}\ar@{.>}[r]
      & **[r]\freepositivecone{\pi_0\sections_{K}(c)}\ar[u]_{\freepositivecone{\pi_0(s\mapsto s(t_c))}},
    }
  \end{equation*}

  These dotted arrows define the components of a cone $\positivebmhomology{1}K\ra\freepositivecone{\pi_0\sections_{K}(-)}$ to the functor $\freepositivecone{\pi_0\sections_{K}(-)}$ from the poset of closed cells in $\timedomain$, reverse ordered by inclusion, to $\CONES$.
  Hence $\lim\freepositivecone{\pi_0\sections_{K}(-)}\neq\varnothing$, hence $\lim\pi_0\sections_{K}(-)\neq\varnothing$ because $\freepositivecone{-}$ is a subfunctor of a limit-preserving functor, and hence $\globalsections{}{\sections_K}\neq\varnothing$ [Lemma \ref{lem:holim}].
\end{proof}

\section{Cohomology}\label{sec:cohomology}
For each paracompact space $X$, take $\cohomology{}{\Grading}X$ to mean the sheaf cohomology
\begin{equation*}
  \cohomology{}{\Grading}X=\cohomology{\closed(X)}{\Grading}(\constantsheaf{X}{\R}).
\end{equation*}

\begin{eg}
  The construction $\cohomology{}{\Grading}$ is equivalently real \v{C}ech cohomology.
\end{eg}

The construction $\cohomology{}{\Grading}X$ is contravariant in spaces $X$.
For a paracompact space $X$ and closed $A\subset X$, the long exact sequence (\ref{eqn:cohomological.les}) specializes to the following long exact sequence, in which the unlabelled arrows are induced from inclusions and restrictions.
\begin{equation}
     \cdots\ra
     \cohomology{}{\Grading}X
      \ra
    \cohomology{}{\Grading}(X\setmin A)
      \xra{\partial^\Grading}
    \cohomology{}{\Grading+1}A
      \ra\cdots,
\end{equation}

For each cofiltered system $X_i$ of paracompact spaces and embeddings between them,
\begin{equation*}
\colim_i\cohomology{}{n-1}X_i=\cohomology{}{n-1}\lim\;\!_iX_i
\end{equation*}
where the limit above is taken in the category of spaces and maps between them \cite[Theorem 10.6]{bredon2012sheaf}.

Positive cones on the cohomologies of oriented manifolds-with-compact-boundary encode information about those orientations as follows.

\begin{defn}
  \label{defn:positive-cohomology-man}
  Define $\positivecohomology{n-1}\manifold$ by the pullback square
  \begin{equation*}
    \xymatrix@C=4pc{
        **[l]\positivecohomology{n-1}\manifold
          \ar[r]\ar[d]\pullbackcorner
      & **[r]\positivebmhomology{0}\partial\manifold\ar[d]
      \\
        **[l]\cohomology{}{n-1}\manifold\ar[r]_{\Delta\circ\cohomology{}{n-1}(\partial\manifold\subset\manifold)}
      & **[r]\bmhomology{0}\partial\manifold,
    }
  \end{equation*}
  where the right vertical arrow is inclusion, natural in non-compact, connected, oriented $n$-manifolds-with-compact-boundary.
\end{defn}

It is straightforward to check that $\positivecohomology{n-1}$ is contravariant in morphisms in $\MANIFOLDS_{n}$.
Moreover, $\positivecohomology{n-1}$ sends cofiltered limits in $\MANIFOLDS_{n}$ to colimits in $\CONES$.
Thus the following extension of $\positivecohomology{n-1}$ is well-defined.

\begin{defn}
  Define $\positivecohomology{n-1}\manifold$ as the following colimit natural in pro-objects $\manifold$ in $\MANIFOLDS_{n}$.
  \begin{equation*}
    \positivecohomology{n-1}\manifold=\colim_i\positivecohomology{n-1}\manifold_i.
  \end{equation*}
\end{defn}

We thus extend the definition of positive cohomology one final time as follows.

\begin{defn}
  \label{defn:positive-cohomology-param}
  Define $\positivecohomology{n-1}\paramspace$ by the pullback square
  \begin{equation*}
    \xymatrix@C=12pc{
        **[l]\positivecohomology{n-1}\paramspace\ar[r]\ar[d]\pullbackcorner
      & **[r]\prod_{x\in T}\positivecohomology{n-1}\paramspace_x\ar[d]
      \\
        **[l]\cohomology{}{n-1}\paramspace\ar[r]_{\prod_{x\in T}\cohomology{}{n-1}(\paramspace_x\subset\paramspace)}
      & **[r]\prod_{x\in T}\cohomology{}{n-1}\paramspace_x,
    }
  \end{equation*}
  where the right vertical arrow is inclusion, natural in spaces $\paramspace$ over $T$ whose fibers have the extra structure of decompositions, in the category of spaces, as limits of pro-objects in $\MANIFOLDS_n$.
\end{defn}

The cone $\positivecohomology{n-1}\paramspace$ inside $\cohomology{}{n-1}\lim_i\paramspace_i$ is always positive.
Hence we call $\positivecohomology{n-1}\paramspace$ the \textit{positive cohomology} of $\paramspace$.

\begin{figure}[h]
    \includegraphics[width=5.5in]{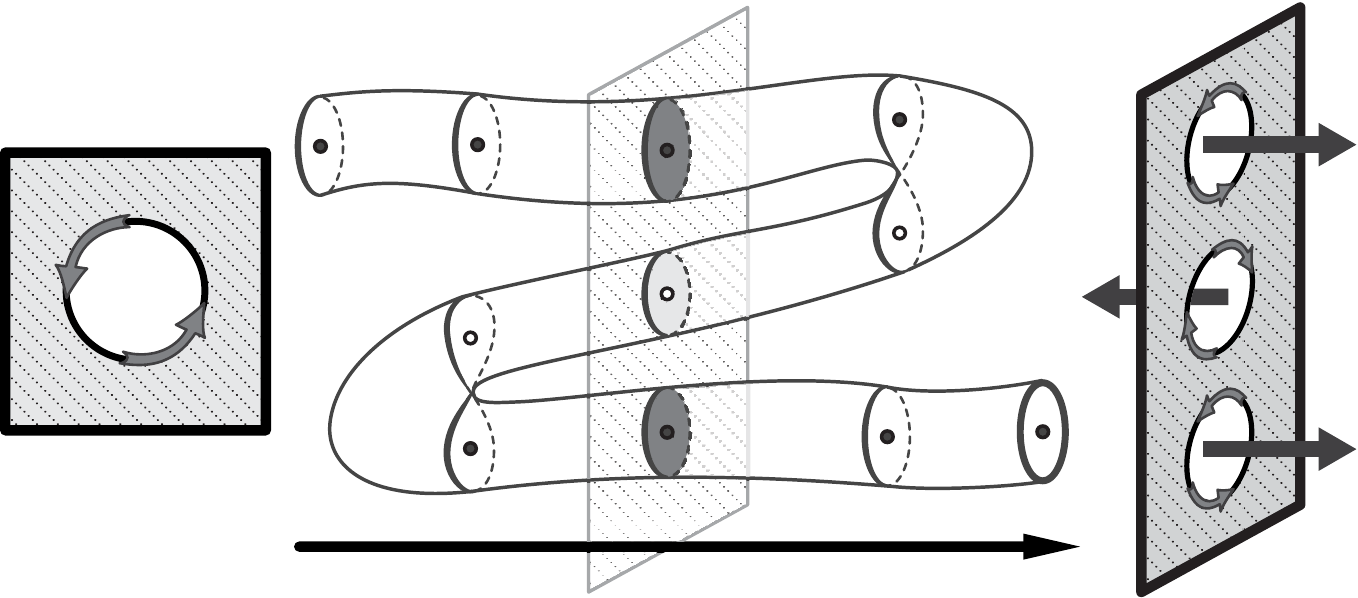}
  \caption{{\bf Cohomological counterparts to $\positivebmhomology{0},\positivebmhomology{1}$}: In the manifold-with-boundary $\manifold$, left, a compactly supported $1$-cohomology class is positive if and only if its restriction to the boundary is a positive multiple of the orientation class of that boundary. In the space $\paramspace$ over $\timedomain$ shown [middle], a $1$-cohomology class is positive if its restriction to each fiber is positive [right].}
  \label{fig:cohomologyorientation}
\end{figure}

\section{Alexander Duality}\label{sec:duality}
\textit{Alexander Duality} $\alexanderduality$ is an isomorphism defined by the commutative diagrams
\begin{equation*}
  \xymatrix@C=7pc{
      **[l]\cohomology{}{\dim\ambientspace-q-1}\sensedspace
      \ar@{.>}[r]^{\alexanderduality}
      \ar[d]_{\partial^{n-1}}
    & **[r]\bmhomology{q}(\ambientspace\setmin\sensedspace)
    \\
      **[l]\cohomology{}{\dim\ambientspace-q}(\ambientspace\setmin\sensedspace)
      \ar[r]_{\Delta}
    & **[r]\bmhomology{q}(\ambientspace\setmin\sensedspace)
      \ar@{=}[u]
  }
\end{equation*}
of isomorphisms, for $q=0,1,2,\ldots,\dim\ambientspace-2$, natural in coverage pairs $(\ambientspace,\sensedspace)$.
Alexander Duality restricts and corestricts to an isomorphism of positive cones.

\begin{thm}
  \label{thm:positive.alexander.duality}
  For coverage pair $(\ambientspace,\sensedspace)$ over $\R^q$ with $\dim E>2$ and $0\leqslant q\leqslant\dim\ambientspace-2$,
  \begin{equation*}
    \positivecohomology{\dim\ambientspace-q-1}\sensedspace\cong\positivebmhomology{q}(\ambientspace\setmin\sensedspace)
  \end{equation*}
  along a restriction and corestriction of Alexander Duality.
\end{thm}
\begin{proof}
  We induct on $q$.
  
  Consider the base case $q=0$.
  It suffices to consider the subcase $q=$ and $\sensedspace$ is an oriented $\dim\ambientspace_t$-submanifold-with-compact-boundary.
  The more general case $q=0$ follows from taking the colimit, over neighborhoods $\manifold$ of $\sensedspace$ in $\ambientspace$ that fall under the aforemntioned special case, of positive Alexander Dualities of the form $\positivecohomology{\dim\ambientspace-1}\manifold\cong\positivebmhomology{0}(\ambientspace\setmin\manifold)$ by definition of $\positivecohomology{\dim\ambientspace-1}\sensedspace$ and $\bmhomology{0}$ compactly supported.
  
  Consider the following commutative diagram of solid arrows, in which the top left square is a pullback square and $\iota_X$ denotes inclusion $\positivebmhomology{0}X\ira\bmhomology{0}X$.
  \begin{equation}
    \label{eqn:2-duality}
    \xymatrix@C=4pc{
          \positivecohomology{\dim\ambientspace-1}\sensedspace
          \ar[r]^{\alpha^+}
          \ar[d]
        & \positivebmhomology{0}\partial\sensedspace
          \ar@{.>}[r]^{\beta^+}
          \ar[d]_{\Delta^{-1}\circ\iota_{\partial\sensedspace}}
        & \positivebmhomology{0}(\ambientspace\setmin\sensedspace)
          \ar[d]_{\iota_{\ambientspace\setmin\sensedspace}}
          \\
          \cohomology{}{\dim\ambientspace-1}\sensedspace
          \ar[r]_-{\cohomology{}{\dim\ambientspace-1}(\partial\sensedspace\subset\sensedspace)}
          \ar[d]_{\partial^{\dim\ambientspace-1}}
        & \cohomology{}{\dim\ambientspace-1}\partial\sensedspace
          \ar[d]^{\partial^{\dim\ambientspace-1}}
          \ar[r]^{\beta}
        & \bmhomology{0}(\ambientspace\setmin\sensedspace)
          \ar@{=}[d]
          \\
          \cohomology{}{\dim\ambientspace}(\ambientspace\setmin\sensedspace)
          \ar@{=}[r]
        & \cohomology{}{\dim\ambientspace}(\ambientspace\setmin\sensedspace)
          \ar[r]_-{\Delta}
        & \bmhomology{0}(\ambientspace\setmin\sensedspace),        
    }
  \end{equation}
  
  The composite of the middle row is Alexander Duality by definition.
  Hence $\cohomology{}{n-1}(\partial\sensedspace\subset\sensedspace)$ is injective.
  Moreover, the sequence
  \begin{equation*}
    \cohomology{}{n-1}\sensedspace\xra{\cohomology{}{n-1}(\partial\sensedspace\subset\sensedspace)}\cohomology{}{n-1}\partial\sensedspace\xra{\partial^{n-1}}\cohomology{\closed(\sensedspace){|\sensedspace\setmin\partial\sensedspace}}{n}\constantsheaf{\sensedspace\setmin\partial\sensedspace}{\R}
  \end{equation*}
  is exact and its last term is trivial by the equalities
  \begin{equation*}
    \cohomology{\closed(C){|\sensedspace\setmin\partial\sensedspace}}{n}\constantsheaf{\sensedspace\setmin\partial\sensedspace}{\R}=\homology{\closed(C){|\sensedspace\setmin\partial\sensedspace}}{0}\constantsheaf{\sensedspace\setmin\partial\sensedspace}{\R}=0,
  \end{equation*}
  the first from Poincar\'{e} Duality and the second from $\sensedspace-\partial\sensedspace$ a filtered union of elements in $\closed(\sensedspace)|(\sensedspace\setmin\partial\sensedspace)$ on which $\bmhomology{0}$ is trivial \cite[Exercise 26]{bredon2012sheaf}.
  Then $\cohomology{}{n-1}(\partial\sensedspace\subset\sensedspace)$, and hence also $\alpha^+$, are isomorphisms.
  
  The arrow $\beta$, whose composite with the isomorphism $\cohomology{}{n-1}(\partial\sensedspace\subset\sensedspace)$ is an isomorphism $\alexanderduality:\cohomology{}{n-1}\sensedspace\ra\bmhomology{0}(\ambientspace\setmin\sensedspace)$, is an isomorphism.
  Moreover, $\beta$ sends an orientation class on a connected component represented by $x$ to the path-component in $\ambientspace\setmin\sensedspace$ whose closure in $\ambientspace$ contains $x$.
  Hence $\beta\Delta^{-1}\iota_\sensedspace$ restricts and corestricts to the bijection
  \begin{equation*}
    \pi_0\partial\sensedspace\cong\pi_0(\ambientspace\setmin\sensedspace),
  \end{equation*}
  and hence the isomorphism $\beta^+$ in (\ref{eqn:2-duality}).
  Hence $\beta^+\alpha^+$, a restriction and corestriction of $\alexanderduality$, is an isomorphism.

  Fix positive integer $Q$.
  Inductively assume the theorem holds for the case $q=Q$.
  Now consider the case $q=Q+1$.

  Consider the following commutative diagram
  \begin{equation}
    \label{eqn:1-duality}
    \xymatrix@C=5pc{
          \positivecohomology{n-1}\sensedspace
          \ar[r]\ar[d]\ar@{.>}[ddr]
        & \prod_{x}\positivecohomology{n-1}\sensedspace_x
          \ar@{.>}[ddr]^{\cong}
          \ar[d]
          \\
          \cohomology{}{n-1}\sensedspace
          \ar[ddr]_-{\alexanderduality}
          \ar[r]_-{\prod_x\cohomology{}{n-1}(C_x\subset C)}
        & \prod_{x}\cohomology{}{n-1}\sensedspace_x
          \ar[ddr]
          \\
        & \positivebmhomology{q}(\ambientspace\setmin\sensedspace)\ar[r]\ar[d]
        & \prod_{x}\positivebmhomology{q-1}(\ambientspace_x-\sensedspace_x)
          \ar[d]
          \\
        & \bmhomology{q}(\ambientspace\setmin\sensedspace)\ar[r]
        & \prod_{x}\bmhomology{q-1}(\ambientspace_x-\sensedspace_x)
    }
  \end{equation}
  of solid arrows, where the front and back squares are pullback squares, the vertical arrows are inclusions, and the right arrow in the bottom square is defined by Alexander Dualities.
  There exists a dotted isomorphism making the side square commute by the inductive assumption.
  The bottom square in (\ref{eqn:1-duality}) commutes by a diagram chase and suitably compatible choices of orientations on $\ambientspace$ and its fibers.
  Hence there exists the desired dotted isomorphism in (\ref{eqn:1-duality}) by universal properties of pullbacks.
\end{proof}

The tools are now assembled to infer the complete criterion for evasion. 
The main theorem of this paper states that the positive cohomology $\positivecohomology{n-1}\sensedspace$ is the complete obstruction to evader capture -- that an evasion path exists if and only if this cohomology is non-empty.  

\begin{cor}
  \label{cor:criterion}
  For a coverage pair $(\ambientspace,\sensedspace)$ over $\R$ with $\dim E>2$,
  \begin{equation}
    \label{eqn:criterion}
    \positivecohomology{\dim\ambientspace-2}\sensedspace\neq\varnothing
  \end{equation}
  if and only if the restriction $\ambientspace\setmin\sensedspace\ra\timedomain$ of the projection $\ambientspace\ra\timedomain$ admits a section.
\end{cor}
\begin{proof}
  Observe that
  \begin{align*}
      \positivecohomology{\dim\ambientspace-2}\sensedspace\neq\varnothing
    & \iff\positivebmhomology{1}(\ambientspace\setmin\sensedspace)\neq\varnothing & \mathrm{Theorem\; \ref{thm:positive.alexander.duality}}\\
    & \iff\ambientspace\setmin\sensedspace\;\text{admits a section} & \mathrm{Lemmas\;\ref{lem:positivity.determines.sections},\ref{lem:sections.determine.positivity}}
  \end{align*}
\end{proof}

\section{Computation}
\label{sec:comput}
We decompose the requisite calculation of positive cohomology in the criterion as a limit of local positive cohomology cones.
Such a limit is naturally described as the global sections of a \textit{cellular sheaf}.
A \textit{cellular sheaf of cones on $\timedomain$} is a functor from the poset of cells in a stratification of $\timedomain$, into vertices and edges, to $\CONES$.
We notate the global sections of a cellular sheaf $\sheaf{3}$ of cones on $\timedomain$ as $\globalsections{}{\sheaf{3}}$; formally
\begin{equation*}
  \globalsections{}{\sheaf{3}}=\lim_c\sheaf{3}(c),
\end{equation*}
where the limit, indexed over all vertices and edges of $\timedomain$, is taken in $\CONES$.
Cellular sheaves are well-suited to applications involving computation, such as the cellular complexes used to compute ordinary homology; see, {\it e.g.}, \cite{CurrySheaves14,GhristElementary14,ShepardCellular85}. 

\begin{eg}
  For a fiberwise oriented space $\paramspace$ over $\timedomain$ equipped with a stratification,
  \begin{equation*}
    \positivecohomology{n}\paramspace_{\openstar(-)}
  \end{equation*}
  defines a cellular sheaf of cones on $\timedomain$, sending each cell $c$ in $\timedomain$ to the positive cohomology $\positivecohomology{n}\paramspace_{\openstar(c)}$ of the fiberwise oriented space $\paramspace_{\openstar(c)}$ over the open star of $c$ in $\timedomain$ and sending each inclusion $v\leqslant e$ of vertex into edge to an appropriate restriction map of cohomologies.
\end{eg}


The criterion for evasion decomposes as the global sections of a cellular sheaf of local positive cohomologies.

\begin{cor}
  \label{cor:local}
  For coverage pair $(\ambientspace,\sensedspace)$ over $\timedomain$ and collection $\mathscr{O}$ of open subsets of $\timedomain$,
  \begin{equation*}
    \positivecohomology{\dim\ambientspace-2}\sensedspace_{\bigcup\mathscr{O}}=\lim_{U\in\mathscr{O}}\positivecohomology{\dim\ambientspace-2}\sensedspace_U.
  \end{equation*}
\end{cor}
\begin{proof}
  There exist natural isomorphisms
  \begin{align*}
      \positivecohomology{n-1}\sensedspace
    & \cong\positivebmhomology{n-1}(\ambientspace\setmin\sensedspace) & \mathrm{Theorem\; \ref{thm:positive.alexander.duality}}\\
    & \cong\lim_U\positivebmhomology{1}(\ambientspace_U\setmin\sensedspace_U) & \mathrm{Lemma\;\ref{lem:positive-exactness}}\\
    & \cong\lim_U\positivecohomology{n-1}\sensedspace_U, & \mathrm{Theorem\; \ref{thm:positive.alexander.duality}}.
  \end{align*}
\end{proof}

Formulated in the language of sheaves,
\begin{equation*}
  \positivecohomology{n-1}\sensedspace=\globalsections{}{\positivecohomology{n-1}\sensedspace_{\openstar(-)}}.
\end{equation*}

Corollary \ref{cor:local} is the starting point for local-to-global calculations of the positive cohomological criterion. 
This allows us to break down $\positivecohomology{n-1}\sensedspace$ into three calculations.
\begin{enumerate}
  \item the cohomology cellular sheaf $\cohomology{}{n-1}\sensedspace_{\openstar(-)}$
  \item the positive cohomology cellular subsheaf $\positivecohomology{n-1}\sensedspace_{\openstar(-)}$
  \item the global sections $\globalsections{}{\positivecohomology{n-1}\sensedspace_{\openstar(-)}}$
\end{enumerate}

The first calculation is purely classical.
The second calculation involves checking which local sections in $\cohomologysheaf{n-1}(U)$ restrict to orientation classes on connected components of $\partial\sensedspace_t$ for all open stars $U$ in a stratification of $\timedomain$ and all $t\in U$ with $\sensedspace_t$ a connected manifold-with-compact-boundary.
The third calculation is nontrivial. 
In this section, we show how to reduce it to a simple linear program.

\begin{prop}
  \label{prop:sheafequalizer}
  For a cellular sheaf $\sheaf{3}$ of cones on $\timedomain$ equipped with a stratification,
  \begin{equation}
    \label{eq:obstructioncomputation}
    \globalsections{}{\sheaf{3}}=
    \prod_{v}\sheaf{3}(v)
    \cap
    \ker\left(
        \prod_{v}\sheaf{3}(v)^{\pm}
        \xra{(\phi_v)_v\mapsto(\sheaf{3}^\pm(e_-\leqslant e)(\phi_{e_-})-\sheaf{3}^\pm(e_+\leqslant e)(\phi_{e_+}))_e}
        \prod_{e}\sheaf{3}(e)^{\pm}
            \right)
            ,
  \end{equation}  
  where $v$ ranges over all vertices and $e$ ranges over all precompact edges in $\timedomain$, and $e_-,e_+$ denote the vertices of a precompact edge in $\timedomain$ having boundaries $e_-<e_+$.
\end{prop}

Otherwise said, global sections ($H^0$) consist of positive local sections (data over $v$) which agree when restricted to data over incident edges, using the kernel to measure agreement.  This is a linear-algebraic view of why the evasion criterion of this paper works whereas others do not: the positive cohomology sheaf $\positivecohomology{n-1}\sensedspace$, instead of computing an Abelian kernel (as in ordinary cohomology), computes the positive kernel of a difference of two maps on positive cones. 

This prompts an algorithmic approach to computation in the cellular setting. 
Let $\mathcal{H}$ be a sheaf of positive, polyhedral cones in finite dimensional real vector spaces and maps between such cones.
Fixing positive bases for the stalks of $\sheaf{3}^\pm$ over vertices and edges, the map
  \begin{equation}
    \label{eq:matrix}
        \prod_{v}\sheaf{3}^\pm(v)
        \ra
        \prod_{e}\sheaf{3}^\pm(e)
            ,
  \end{equation}  
in (\ref{eq:obstructioncomputation}) yields a real matrix (a cellular coboundary matrix). 
The entries of this matrix are differences, and thus may be positive or negative. However, to contribute to $H^0$, the kernel must have nontrivial intersection with $\prod_v\sheaf{3}(v)$, meaning that one must compute the kernel subject to a sequence of inequalities. 

Computing $\globalsections{}{\sheaf{3}}$ is therefore reduced to the question of whether a known subspace of a real vector space (the kernel of (\ref{eq:matrix})) intersects a positive cone in that vector space.  By normalizing the cone to, say, unit $\ell^1$ distance to the origin, finding a nonzero element of $\globalsections{}{\sheaf{3}}$ is expressible as a simple linear programming problem. Assuming a fixed bound on the complexity of the cones (the number of constraints), the solution to this problem via standard methods is linear in the number of variables \cite{MegiddoLinear84}, in this case, the sum of the number of generators of $\sheaf{3}(v_i)$ over all $i$. 

\section{Planar examples}\label{sec:examples}

We give a computation of the criterion for examples in the planar case ($n=2$).
The two principal examples from \cite{ACEvasion14} alluded to in \S\ref{sec:intro} are the simplest examples with which to show how $\positivecohomology{1}$ determines evasion. Consider the coverage regions $\sensedspace$ over $\timedomain$ illustrated in Figures \ref{fig:works} and \ref{fig:noworks}. Stratify $\timedomain$ as shown in a manner compatible with the projection map of $\sensedspace$ to $\timedomain$: this decomposition has four vertices $\{v_i\}_1^4$ and five edges $\{e_j\}_1^5$ with $e_1$ and $e_5$ extending to $\pm\infty$. 

\begin{figure}[h]
      \includegraphics[width=5.0in]{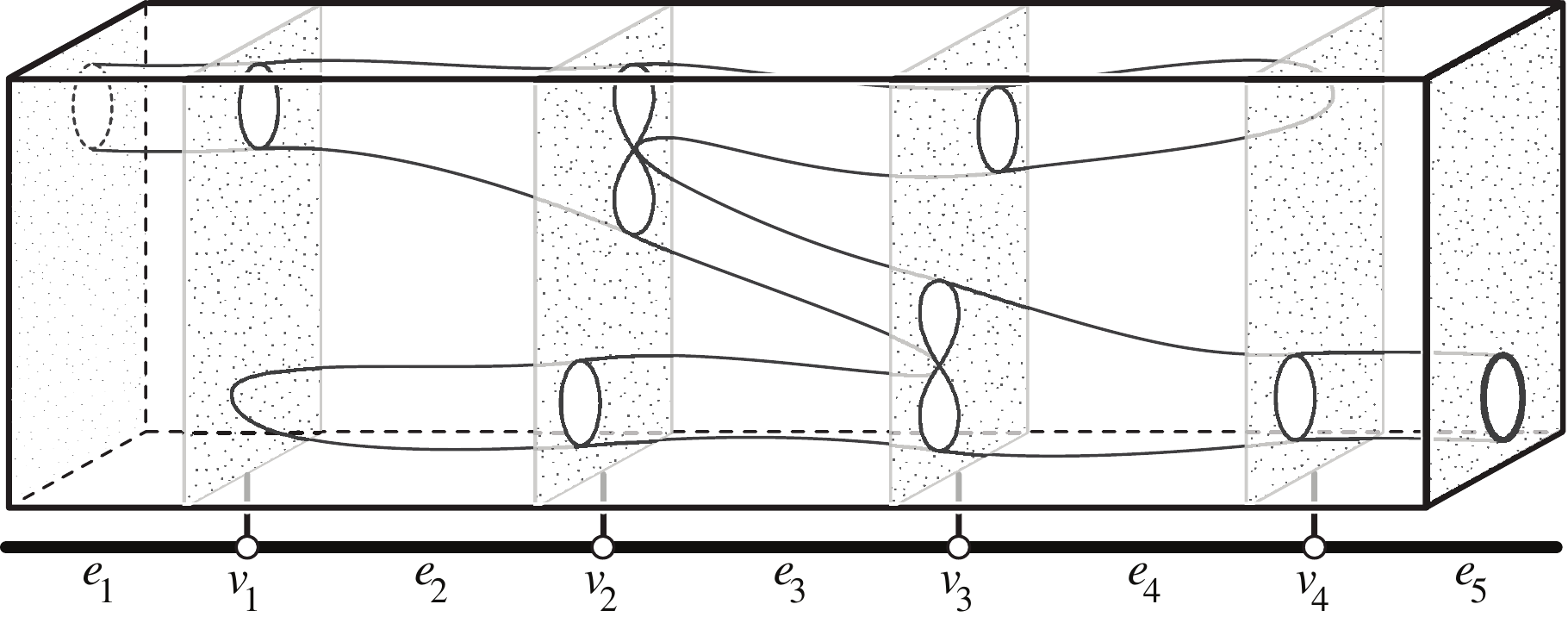} 
  \caption{A simple planar case in which an evasion path exists. Shown is a parameterized space of (unbounded, shaded) sensed regions, $\sensedspace_t$ and (bounded) evasion sets as a function of time.}
  \label{fig:works}
\end{figure}

\begin{figure}[h]
      \includegraphics[width=5.0in]{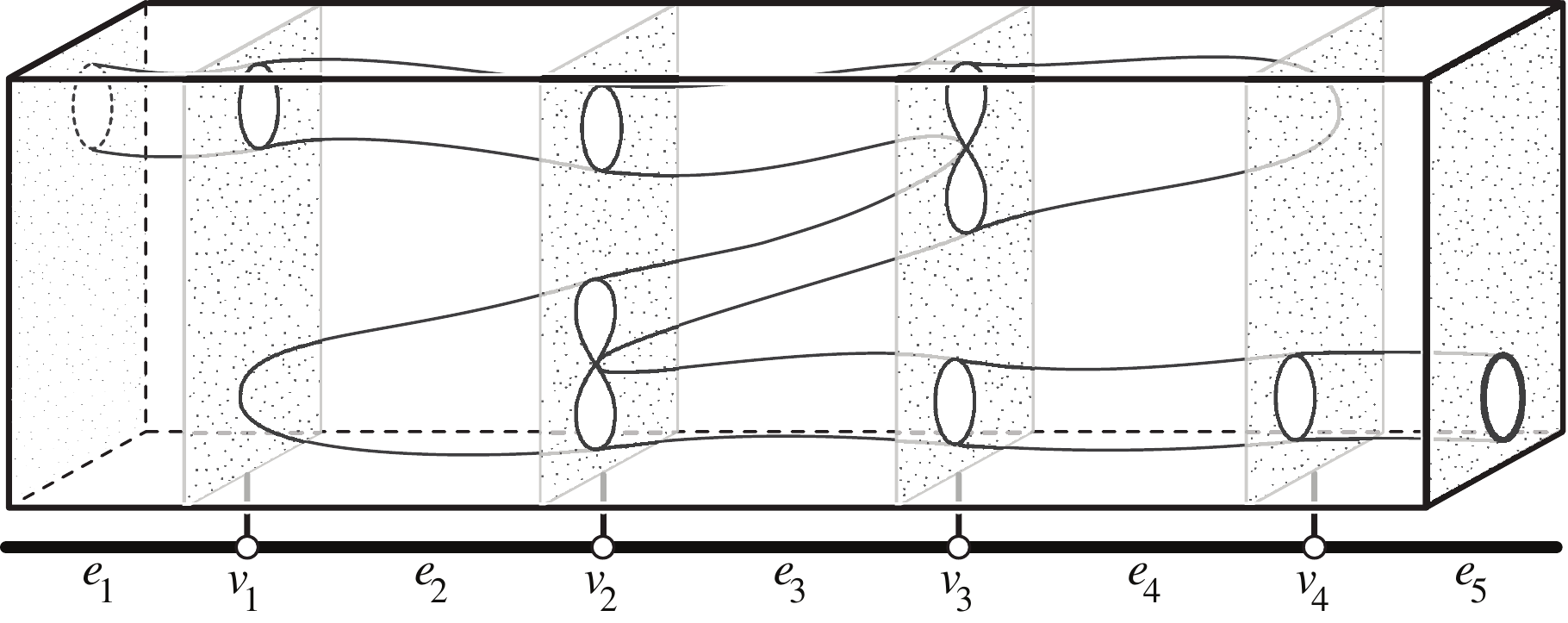} 
  \caption{The analogue of Figure \ref{fig:works} for which no evasion path exists.}
  \label{fig:noworks}
\end{figure}

With this subdivision of $\timedomain$, the associated sheaves are cellular, with stalks computed in terms of local sections. (Specifically, the stalk over a vertex $v_i$ equals the value of the sheaf on the open star $\openstar(v_i)$.) In both cases, the dimension of the stalks of $\positivecohomologysheaf{1}$ are given as follows: 
\begin{itemize}
\item There is one generator over $e_1$ and $e_5$; two over $e_2$ and $e_4$; and three over $e_3$.
\item There is one generator over $v_1$ and $v_4$; and three over $v_2$ and $v_3$. 
\end{itemize}
Label all these generators by the cells in the time axis, with superscripts where "t" connotes a generator associated to the top hole, "b" the bottom, and "m" the middle. In all cases, the cones determined by $\positivehomologysheaf{1}$ are positive orthants in $\R^k$ for $k$ the number of generators of the stalk. 

All the induced maps from stalks over $v_i$ to stalks over $e_i$ and $e_{i+1}$ are either trivial inclusions or projections given by correspondence of generators, and it is in this regard that the sheaves for Figures \ref{fig:works} and \ref{fig:noworks} differ.  
The global criterion $\globalsections{}{\positivecohomologysheaf{1}}$ is computed via Equation (\ref{eq:obstructioncomputation}) using these restriction maps with the appropriate signs. 
Note that, since we are computing cellular cohomology, we must zero-out all the maps to $e_1$ and $e_5$, since these have noncompact closure in $\R$, as per the discussion above. 
We present this data in the form of a matrix with entries of the form $a^+_{ij}-a^-_{ij}$, so as to see the explicit dependence on whether the term comes from the left or right in time. 
For the system of Figure \ref{fig:works}, this matrix is:
\begin{equation}
\label{eq:works}
\prod_{v}\positivecohomology{1}(\delta_v^+) - \positivecohomology{1}(\delta_v^-)
=
\begin{array}{|c||c|c|c|c|c|c|c|c|}
  \hline
          & v_1^t    & v_2^t      & v_2^m      & v_2^b      & v_3^t      & v_3^m      & v_3^b      & v_4^b \\  \hline \hline
  e_2^t & 0\setmin 1 & 1\setmin 0 & 1\setmin 0 & 0\setmin 0 & 0\setmin 0 & 0\setmin 0 & 0\setmin 0 & 0\setmin 0 \\    \hline
  e_2^b & 0\setmin 0 & 0\setmin 0 & 0\setmin 0 & 1\setmin 0 & 0\setmin 0 & 0\setmin 0 & 0\setmin 0 & 0\setmin 0 \\  \hline
  e_3^t & 0\setmin 0 & 0\setmin 1 & 0\setmin 0 & 0\setmin 0 & 1\setmin 0 & 0\setmin 0 & 0\setmin 0  & 0\setmin 0 \\  \hline
  e_3^m & 0\setmin 0 & 0\setmin 0 & 0\setmin 1 & 0\setmin 0 & 0\setmin 0 & 1\setmin 0 & 0\setmin 0  & 0\setmin 0 \\  \hline
  e_3^b & 0\setmin 0 & 0\setmin 0 & 0\setmin 0 & 0\setmin 1 & 0\setmin 0 & 0\setmin 0 & 1\setmin 0  & 0\setmin 0 \\  \hline
  e_4^t & 0\setmin 0 & 0\setmin 0 & 0\setmin 0 & 0\setmin 0 & 0\setmin 1 & 0\setmin 0 & 0\setmin 0  & 0\setmin 0 \\  \hline
  e_4^b & 0\setmin 0 & 0\setmin 0 & 0\setmin 0 & 0\setmin 0 & 0\setmin 0 & 0\setmin 1 & 0\setmin 1  & 1\setmin 0 \\  \hline
\end{array}
\end{equation}

One computes that $\positivecohomology{1}\sensedspace\neq\emptyset$ as follows. The kernel of (\ref{eq:works}) is one-dimensional with a positive generator given by the sum of generators $v_1^t + v_2^m + v_3^m + v_4^b$.  These generators, of course, correspond to the dual evasion path sequence easily observed from Figure \ref{fig:works}. The cohomology $\positivecohomology{1}\sensedspace$ is thus isomorphic to an open ray. In this example, one can observe directly that the kernel intersects the positive orthant, and no linear programming is necessary. A larger, more opaque example would lead to more complex computation. 

For Figure \ref{fig:noworks}, the positive-negative coboundary matrix on the sheaf $\positivecohomologysheaf{1}$ is:
\begin{equation}
\label{eq:noworks}
\prod_{v}\positivecohomology{1}(\delta_v^+) - \positivecohomology{1}(\delta_v^-)
=
\begin{array}{|c||c|c|c|c|c|c|c|c|}
  \hline
          & v_1^t    & v_2^t      & v_2^m      & v_2^b      & v_3^t      & v_3^m      & v_3^b      & v_4^b \\  \hline \hline
  e_2^t & 0\setmin 1 & 1\setmin 0 & 0\setmin 0 & 0\setmin 0 & 0\setmin 0 & 0\setmin 0 & 0\setmin 0 & 0\setmin 0 \\    \hline
  e_2^b & 0\setmin 0 & 0\setmin 0 & 1\setmin 0 & 1\setmin 0 & 0\setmin 0 & 0\setmin 0 & 0\setmin 0 & 0\setmin 0 \\  \hline
  e_3^t & 0\setmin 0 & 0\setmin 1 & 0\setmin 0 & 0\setmin 0 & 1\setmin 0 & 0\setmin 0 & 0\setmin 0  & 0\setmin 0 \\  \hline
  e_3^m & 0\setmin 0 & 0\setmin 0 & 0\setmin 1 & 0\setmin 0 & 0\setmin 0 & 1\setmin 0 & 0\setmin 0  & 0\setmin 0 \\  \hline
  e_3^b & 0\setmin 0 & 0\setmin 0 & 0\setmin 0 & 0\setmin 1 & 0\setmin 0 & 0\setmin 0 & 1\setmin 0  & 0\setmin 0 \\  \hline
  e_4^t & 0\setmin 0 & 0\setmin 0 & 0\setmin 0 & 0\setmin 0 & 0\setmin 1 & 0\setmin 1 & 0\setmin 0  & 0\setmin 0 \\  \hline
  e_4^b & 0\setmin 0 & 0\setmin 0 & 0\setmin 0 & 0\setmin 0 & 0\setmin 0 & 0\setmin 0 & 0\setmin 1  & 1\setmin 0 \\  \hline
\end{array}
\end{equation}

One computes that this matrix has, as before, one-dimensional kernel; however, the generator of this kernel is 
\begin{equation*}
  v_1^t + v_2^t + v_3^t - v_3^m - v_2^m + v_2^b + v_3^b + v^4_b,    
\end{equation*}
which has both positive and negative terms. The intersection of this kernel with the open positive orthant is empty; therefore,  $\positivecohomology{1}\sensedspace=\emptyset$, and there is no evasion path. 

\section{Concluding remarks}
\label{sec:conc}

\begin{enumerate}
\item
For genuine applications to pursuit problems, much of the structure
here introduced is adaptable.  For example, the presence of obstacles
in the free space (an important consideration in practice) is
admissible by including the obstacle regions within the coverage
domain. The use of a sensor network to determine the coverage region
(by means of, say, a Vietoris-Rips complex) may be likewise permissible,
if the relevant cohomology and orientation data are discernible. 
This is an interesting open problem.
\item
This work has considered only the connectivity data about the evasion
paths. The full space of evasion paths may have interesting
topological features beyond connectivity alone.  There would appear also to
be virtue in augmenting $\pathcomponents$ with additional data about evaders:
number, identity, team-membership, or other features, much in the same
way that {\it feature vectors} are used in target tracking.  It is to
be suspected that sheaf constructions permit a great deal of such
information to be encoded.
\item
The use of linear programming to compute the positive cohomology in
this paper is perhaps the beginnings of a broader set of techniques in
computational equalizers and what we would call {\it homological programming}.
\item
Finally, the utility of directed co/homology appears to be relevant to
other applications in which directedness is crucial, including
Bayesian belief networks.
\end{enumerate}

\appendix

\addcontentsline{toc}{section}{Appendix}
\addtocontents{toc}{\protect\setcounter{tocdepth}{0}}

\section{Existence of sections}\label{appendix:sec:parametrized.spaces}
We give homotopical criteria for open subspaces of vector bundles over $\timedomain$ to admit sections.

\begin{lem}
  \label{lem:holim}
  Consider the following data.
  \begin{enumerate}
    \item a CW structure on $\timedomain$; and
    \item an open subspace $\paramspace$ of a Euclidean space over $\timedomain$
  \end{enumerate}
  The natural function $\sections_\paramspace(\timedomain)\ra\lim_c\pathcomponents\sections_\paramspace(c)$, where $c$ denotes a closed cell in $\timedomain$, is surjective.
\end{lem}

\begin{proof}
  Fix a natural number $n$.
  We take $\paramspace$ to be an open subspace of $\R^{n+1}$ such that $\projection_\paramspace$ is the projection onto the last coordinate. Let $\xi$ be projection $\paramspace\ra\R^n$ onto the first $n$ coordinates. We take the vertices of $\timedomain$ to be the integers without loss of generality.

  Let $i$ denote an integer, and consider $s_i\in\sections_\paramspace([i,{i+1}])$ for each $i$ and path $\gamma_i:s_{i-1}({i})\leadsto s_{i}({i})$ for each $i$. It suffices to construct $s\in\sections_\paramspace(\timedomain)$ with $s_{|(i,{i+1})}\sim s_i$ for each $i$.

  Fix $i$.

  The subspace $\im\;\gamma_i\subset\paramspace$ admits an open cover $\mathscr{O}$ consisting of open hyperrectangles because $\R^n$, and hence also its open subspace $\paramspace$, admit open bases consisting of open hyperrectangles.
  We can take $\mathscr{O}$ to be finite by $[i,i+1]$ and hence $\im\;\gamma_i$ compact.
  Hence $\bigcap_{U\in\mathscr{O}}\projection_\paramspace(U)$ is an open neighborhood in $\timedomain$ of $i$, and hence contains a subset of the form $[i,i+2\epsilon_i]$ for $0<\epsilon_i<\half$.  
  Hence 
  \begin{equation*}
    \paramspace\supset\bigcup\mathscr{O}\supset\bigcup_{V\in\mathscr{O}}\xi(V)\times\projection_P(V)\supset\bigcup_{V\in\mathscr{O}}\xi(V)\times\bigcap_{U\in\mathscr{O}}\projection_PU\supset\xi(\im\;\gamma_i)\times[i,i+2\epsilon_i]
  \end{equation*}

  Hence we can define a section $s:\timedomain\ra\paramspace$ to $\paramspace$ by
  \begin{equation*}
    s(x)=
    \begin{cases}
      (\xi(\gamma_i(\nicefrac{(x-i)}{\epsilon_i}),x), & x\in[i,i+\epsilon_i]\\
      (\xi(s_i(2x-2\epsilon_i-i)),x), & x\in[i+\epsilon_i,i+2\epsilon_i]\\
      s_i(x) & x\in[i+2\epsilon_i,i+1]
    \end{cases}
  \end{equation*}
  
  For each $i\in\Z$, $s_{|[i,i+1]}\sim s_i$ along the homotopy $h_i$ defined by
  \begin{equation*}
    h_i(x,t)=
    \begin{cases}
      (\xi(\gamma_i(1-t+\nicefrac{(x-i)}{\epsilon_i}),x), & x\in[i,i+t\epsilon_i]\\
      (\xi(s_i(2x-2t\epsilon_i-i)),x), & x\in[i+t\epsilon_i,i+2t\epsilon_i]\\
      s_i(x) & x\in[i+2t\epsilon_i,i+1]
    \end{cases}
  \end{equation*}
\end{proof}

\bibliography{pursuitevasionrefs}{}
\bibliographystyle{siam}
\end{document}